\documentclass[12pt]{article}
  \usepackage{amsfonts,amsmath}
  \usepackage{latexsym}
  \usepackage[ansinew]{inputenc}

   \def\R{{\rm I\mskip -3.5mu R}}
   \def\N{\mathbb{N}}
    \def\Z{\mathbb{Z}}

   \newenvironment{proof}
{\medskip\noindent{\bf Proof.\/}} {\null \hfill $\Box$
\par\medskip}
   \newenvironment{proof1}
{\medskip\noindent{\bf Proof of Theorem 2.1\/}} {\null \hfill $\Box$
\par\medskip}

\newtheorem{proposition}{Proposition}[section]

\newtheorem{theorem}{Theorem}[section]
\newtheorem{lemma}{Lemma}[section]
\newtheorem{remark}{Remark}[section]

\newcommand\AMSname{AMS subject classifications}

\title{ The effect of a discontinuous weight for a critical Sobolev problem}
\author{Rejeb Hadiji\thanks{Universit\'e Paris-Est Cr\'eteil, Laboratoire d'Analyse et
de Math\'ematiques Appliqu\'ees, CNRS UMR 8050, UFR des Sciences et
Technologie, 61, Avenue du G\'en\'eral de Gaulle B\^at. P3, 4e
\'etage, 94010 Cr\'eteil Cedex, France. E-mail :
hadiji@u-pec.fr }  $\,$ Sami Baraket\thanks{D\'epartement de Math\'ematiques,
Facult\'e des Sciences de Tunis,
Campus Universitaire, 2092 Tunis,
Universit\'e Tunis El Manar,
Tunisie.
E-mail: smbaraket@yahoo.fr} $\,$  and Habib Yazidi
\thanks{Universit\'e de Tunis, Ecole Nationale Sup\'erieure d'Ing\'enieurs de Tunis, 5 Avenue Taha Hssine, Bab Mnar 1008 Tunis,  Tunisie. E-mail : habib.yazidi@gmail.com} }

\begin{document}

\date{}

\maketitle

\begin{abstract}
We study the minimizing problem
$\inf\left\{\displaystyle\int_{\Omega}p(x)|\nabla u|^{2}dx,\,u\in
H^{1}_{0}(\Omega),\,\|u\|_{L^{\frac{2N}{N-2}}(\Omega)}=1\right\}$
where $\Omega$ is a smooth bounded domain of $\R^{N}$, $N\geq 3$ and
$p$ a positive discontinuous function. We prove the existence of a minimizer under some assumptions.

\medskip

\medskip

\par
\noindent2010 \AMSname: 35J20, 35J25, 35H30, 35J60.
\end{abstract}
\section{Introduction}
Let $\Omega$ be a smooth bounded domain of $\R^{N}$, $N\geq 3$, $2^{*}=\frac{2N}{N-2}$ the critical exponent for Sobolev embedding. Define $\Omega_{1}$ and $\Omega_{2}$ two disjoint domains such that $\bar{\Omega}=\bar{\Omega}_{1}\cup \bar{\Omega_{2}}$. Denote $\Gamma=\partial\Omega_{1}\cap \partial\Omega_{2}$. Define the set $V(\Omega)=\displaystyle\left\{ u\in H^{1}_{0}(\Omega),\quad \int_{\Omega}|u|^{2^{*}}dx=1\right\}$ and define the barycenter function, see \cite{C}, $\beta:\, V(\Omega)\rightarrow \R^{N}$, $u\rightarrow \displaystyle\int_{\Omega}x\,|u|^{2^{*}}dx$.\\
We consider the minimizing problem
\begin{equation}
S(p)=\displaystyle\inf_{u\in V(\Omega),\,\,
\beta(u)\in \Gamma} \displaystyle\int_{\Omega}p(x)|\nabla u|^{2}dx,
\label{eq5}
\end{equation}
where $p$ is a discontinuous function as the following
\begin{equation}
p(x)=\left\{\begin{array}{lll} \alpha_1\quad &\textrm{if $x\in\Omega_{1}$}\\[\medskipamount]
\alpha_2 \quad &\textrm{if $x\in {\bar{\Omega}}_{2}$},
\end{array}
\right.
\label{condition poids}
\end{equation}
where $\alpha_{i}$, $i=1,2$ are some positive constants such that $\alpha_{1}<\alpha_{2}$.
It's important to remark that without the additional condition $\beta(u)\in \Gamma$ we have $S(p)=\alpha_{1} S$, as one can verify concentrating an extremal function for the best Sobolev constant $S$ near a point in the interior of the region $\Omega_{1}$. In this case, the infimum $S(p)$ is not achieved\\
This problem is closely related to the best constant in Sobolev inequality in $\R^{N}$. It posses many interesting properties, see Talenti \cite{T}, and arising in many areas of mathematics and in a geometric context namely for example in the Yamabe problem and the prescribed scalar
curvature problem see Aubin \cite{A}. It's invariance under dilations produces a lack of compactness.\\The phenomenon of lack of compactness and the failure of the Palais Smale condition of this type of problem has been the subject of several studies and it was analyzed in minute detail by Struwe \cite{S}.

In the case where $p$ is a constant, it is well known that (\ref{eq5}) is not achieved for a general domain $\Omega$. Nevertheless, Brezis and Nirenberg showed in \cite{BN} that (\ref{eq5}) has minimizer under a linear perturbation, also, in the same sprit, Demyanov and Nazarov establish, in \cite{DN}, sufficient conditions for the existence of an extremal function in four embedding theorems for more general Sobolev spaces. Bahri and Coron in \cite{BC}
 proved that the Euler equation associated to this problem is solvable when some homology group of the domain with coefficients in $\Z/2\Z$ is  nontrivial, see also the work of Coron \cite{C}.
  In the case where $p$ is a smooth positive function, Hadiji and Yazidi proved that the study of problem (\ref{eq5}) depends on the behavior of the weight $p$ near its minima, see \cite{HY}, see also Hadiji, Molle, Passaseo and Yazidi \cite{HPMY}) and Furtado and Souza\cite{F}.\\
 One may ask whether the lack of compactness of the variational functional associated to  (\ref{eq5}) can be made up by the discontinuity of the weight.\\
 In this work, we investigate the effect of non smoothness of weight $p$ on the existence of solution without linear perturbation or additional conditions on the domain $\Omega$.
\section{Statements and proofs of results}
Let
\begin{equation*}
S_{\alpha_1,\,\alpha_2}=\inf\left\{\alpha_1 \int_{\R^{N}_{+}}|\nabla u|^{2}dx+\alpha_2\int_{R^{N}_{-}}|\nabla u|^{2}dx,\,\,\, u\in H^{1}(\R^{N}),\,\,\, \textrm{$u\neq0$\,in $\R^{N}_{\pm}$,}\,\,\, \|u\|_{L^{2^{*}}(\R^{N})}=1\right\},
\end{equation*}
where $\R^{N}_{+}=\left\{(x',x_{N}) \in \R^{N-1}\times [0,\,\infty[\right\}$ and $\R^{N}_{-}=\left\{(x',x_{N}) \in \R^{N-1}\times ]-\infty,\,0]\right\}$.\\
Set
$$S^{+}=\inf\left\{\int_{\R^{N}_{+}}|\nabla u|^{2}dx,\quad u\in H^{1}(\R^{N}_{+}),\,\,\, \textrm{$u\neq0$\,in $\R^{N}_{+}$,}\,\,\,\, \|u\|_{L^{2^{*}}(\R^{N}_{+})}=1\right\}.$$

and
$$S^{-}=\inf\left\{\int_{\R^{N}_{-}}|\nabla u|^{2}dx,\quad u\in H^{1}(\R^{N}_{-}),\,\,\, \textrm{$u\neq0$\,in $\R^{N}_{-}$,}\,\,\,\, \|u\|_{L^{2^{*}}(\R^{N}_{-})}=1\right\}.$$
It is easy to verify that (see for example \cite{CP})
\begin{equation}
S^{+}=S^{-}=\frac{S}{2^{\frac{2}{N}}},
\label{eq2}
\end{equation}
where $S$ is the best constant of the Sobolev embedding defined by
$$S=\displaystyle\inf_{u\in H^{1}(\R^{N})\setminus\{0\}}\frac{\displaystyle\int_{\R^{N}}|\nabla u|^{2}dx}{\displaystyle\left(\int_{\R^{N}}|u|^{2^{*}}dx\right)^{\frac{2}{2^{*}}}} .$$
We state our main results as follow

\begin{theorem}
The following equality holds
$$S_{\alpha_1,\,\alpha_2}=\left(\frac{\alpha_1^{\frac{N}{2}}+\alpha_2^{\frac{N}{2}}}{2} \right)^{\frac{2}{N}}S.$$
\label{th1}
\end{theorem}
\begin{theorem}
Let $\Omega$, $\Omega_{1}$, $\Omega_{2}$, $p$ be as defined in the introduction and let $x_{0} \in \Gamma$. Assume that the following geometrical condition (g.c.) on $\Gamma$ holds: in a neighborhood of $x_{0}$ , $\Omega_{2}$ lies on one side of the tangent plane at $x_{0}$ and the mean curvature with respect to the unit inner normal of $\Omega_{2}$ at
$x_{0}$ is positive.
Then $S(p)$ is attained by some $u\in H_{0}^{1}(\Omega)$ .
\label{th2}
\end{theorem}
The following proposition presents a strict lower bound for the minimizing problem
\begin{proposition}
The following inequality holds
 $$\alpha_{1}\,S< S(p).$$
\end{proposition}
\begin{proof}
We have $S(p)\geq \alpha_{1}\,S$. Arguing by contradiction, suppose that $S(p)=\alpha_{1}\,S$. Let $\{u_{j}\}$  {be} a minimising sequence, that is, for every $j\in \N$, $u_{j} \in V(\Omega)$, $\beta(u_{j})\in \Gamma$ and $\displaystyle\lim_{j\rightarrow +\infty}\int_{\Omega} p(x)|\nabla u_{j}|^{2}dx=\alpha_{1}\,S$.\\
Since $\displaystyle\int_{\Omega} p(x)|\nabla u_{j}|^{2}dx\geq \alpha_{1}\,S$ then $\displaystyle\lim_{j\rightarrow +\infty}\int_{\Omega} |\nabla u_{j}|^{2}dx=S$. Therefore, there exists $z_{0}\in \bar{\Omega}$ such that, for a subsequence, $|\nabla u_{j}|^{2}\rightarrow S \, \delta_{z_{0}}$ and $|u_{j}|^{2^{*}}\rightarrow\delta_{z_{0}}$, where $\delta_{z_{0}}$ is the Dirac mass in $z_{0}$, see \cite{L}.\\
Since $\beta(u_{j})\in \Gamma$ for every $j\in \N$, it follows that $z_{0}\in \Gamma$ and $p(z_{0})=\alpha_{2}> \alpha_{1}$. Therefore $\displaystyle\lim_{j\rightarrow +\infty}\int_{\Omega} p(x)|\nabla u_{j}|^{2}dx=p(z_{0})\,S>\alpha_{1}\,S$, which gives a contradiction.
\end{proof}
\begin{remark} Let us give  simple  examples for which the condition $(g.c.)$ in Theorem 2 is fulfilled or not.
Let $\Omega = B(0,R)$, $R>1$ and $e_1 =(1,0...,0)$.\par
Set $\Omega_2 = B(e_1,R) \cap \Omega$, $\Omega_1 =  \Omega\setminus \overline{\Omega}_2$, $\Gamma =\partial\Omega_{1}\cap\partial\Omega_{2}$ and $x_{0}$ in the interior of $\Gamma$. We have condition $(g.c.) $ holds.
\par
For  $\Omega_1 = B(e_1,R) \cap \Omega$, $\Omega_2 =
\Omega\setminus \overline{\Omega}_1$ and , $\Gamma =\partial\Omega_{1}\cap\partial\Omega_{2}$ and $x_{0}$ in the interior of $\Gamma$. We have condition $(g.c.)
$ does not hold. More precisely, in any neighborhood of $x_ {0}$,
$\Omega_{2}$ does not lie on one side of the tangent plane at $x_
{0}$ and the mean curvature with respect to the unit inner normal of
$\Omega_{2}$ at $x_{0}$ is negative.
\par
Let $\Omega_1 =  \{ (x_1,...,x_N) \in \Omega \,\,\,\,\,\, s.t. \quad x_1 > 0 \}$ and $\Omega_2 = \{ (x_1,...,x_N)\in \Omega \,\,\,\,\,\, s.t. \quad x_1< 0\}$ and $x_{0}=0$. We have condition $(g.c.) $ hold,
more precisely, in any neighborhood of $0$, $\Omega_{2}$  lies on one side of the tangent plane at $0$ but the mean curvature with respect to the unit inner normal of $\Omega_{2}$ at  $0$ is 0.

\end{remark}

If $\Gamma$  is flat, that is mean that mean curvature at any point of  $\Gamma$ is zero, then we have the following non-existence result:
\begin{proposition}
Let $\Omega=B(0,R)$, $\Gamma =\{x \in \Omega \,\,\setminus x_{N}=0\}$ divided $\Omega$ into two subdomains $\Omega_{1}$ and $\Omega_{2}$. For $i=1,2$, let $p_i(x) = \alpha_i$ in $\Omega_i$ with $\alpha_1<\alpha_2$. Then $S(p)$ is not achieved.
\end{proposition}
Indeed, we have, if (\ref{eq5})  is achieved by $u$   then $ \vert u \vert$ is a minimization solution of (\ref{eq5}) . Let us suppose that $S(p)$ is achieved by some positive function $u \geq 0$,
then there exists a Lagrange multiplier $\mu\in\, R$ such that $u$ satisfies the Euler equation
\begin{equation}
\left\{
\begin{array}{llllll}
-\alpha_1\Delta u=\mu\,u^{2^{*}-1}&\textrm{in $\Omega_{1}$},\\[\medskipamount]
-\alpha_2\Delta u=\mu\,u^{2^{*}-1}&\textrm{in $\Omega_{2}$},\\[\medskipamount]
\alpha_1\frac{\partial u}{\partial \nu_{1}}+\alpha_2\frac{\partial u}{\partial \nu_{2}}=0&\textrm{on $\Gamma$},\\[\medskipamount]
u\not = 0 &\textrm{on $\Gamma$}\\[\medskipamount]
u=0&\textrm{on $\partial\Omega$},
\end{array}
\right.
\label{poh1}
\end{equation}
where $\nu_{1}$ and $\nu_{2}$ are respectively the outward normal of $\Omega_{1}$ and $\Omega_{2}$.\\
On one hand we multiply (\ref{poh1}) by $\nabla u\cdot x$ and we  integrate. On the other hand we multiply (\ref{poh1}) by $\frac{N-2}{2}u$ and we integrate. We obtain, after some computations, the Pohozaev identity
$$-\int_{\Gamma}\left[\alpha_1(x\cdot \nu_{1})|\frac{\partial u}{\partial \nu_{1}}|^{2}+\alpha_2(x\cdot \nu_{2})|\frac{\partial u}{\partial \nu_{2}}|^{2}\right]ds_{x}=\int_{\partial\Omega\cap\partial\Omega_{1}}\alpha_1\,(x\cdot\nu)|\frac{\partial u}
{\partial \nu}|^{2}+\alpha_2\,(x\cdot\nu)|\frac{\partial u}{\partial \nu}|^{2}ds_{x},$$
where $\nu$ is the outward of $\partial\Omega$. Since $B(0,R)$ is star-shaped about $0$ then $x\cdot \nu >0$ on $\partial\Omega$ and then
$$-\int_{\Gamma}\left[\alpha_1(x\cdot \nu_{1})|\frac{\partial u}{\partial \nu_{1}}|^{2}+\alpha_2(x\cdot \nu_{2})|\frac{\partial u}{\partial \nu_{2}}|^{2}\right]ds_{x}>0$$
which gives a contradiction since $x\cdot \nu_{1}=x\cdot \nu_{2}=0$ for every $x$ in  $\Gamma$. Therefore $S(p)$ is not achieved.

\begin{proof1}
On one hand, we claim that
\begin{equation}
S_{\alpha_1,\,\alpha_2}\geq
\left(\frac{\alpha_1^{\frac{N}{2}}+\alpha_2^{\frac{N}{2}}}{2}
\right)^{\frac{2}{N}}S. \label{eq3}
\end{equation}
Indeed, we see that, for all $t\in]0,\,1[$ we have
\begin{equation}
\begin{array}{lll}
S_{\alpha_1,\,\alpha_2}=\\
\inf\left\{\alpha_1 \int_{\R^{N}_{+}}|\nabla
u|^{2}dx+\alpha_2\int_{R^{N}_{-}}|\nabla u|^{2}dx,\,\, u\in
H^{1}(\R^{N}),\,\,
\|u\|^{2^{*}}_{L^{2^{*}}(\R^{N}_{+})}=t,\,\,\|u\|^{2^{*}}_{L^{2^{*}}(\R^{N}_{-})}=1-t\right\}.
\label{eq4}
\end{array}
\end{equation}
Therefore
\begin{eqnarray}
S_{\alpha_1,\,\alpha_2}&\geq \alpha_1\inf\left\{ \int_{\R^{N}_{+}}|\nabla u|^{2}dx,\quad u\in
H^{1}(\R^{N}),\,\,\,\,
\|u\|^{2^{*}}_{L^{2^{*}}(\R^{N}_{+})}=t,\,\,\,\,\|u\|^{2^{*}}_{L^{2^{*}}(\R^{N}_{-})}=1-t\right\}\nonumber\\[\medskipamount]
&+\alpha_2\inf\left\{\int_{R^{N}_{-}}|\nabla u|^{2}dx,\quad u\in
H^{1}(\R^{N}),\,\,\,\,
\|u\|^{2^{*}}_{L^{2^{*}}(\R^{N}_{+})}=t,\,\,\,\,\|u\|^{2^{*}}_{L^{2^{*}}(\R^{N}_{-})}=1-t\right\}.
\label{Avril1}
\end{eqnarray}
At this stage, define
$$A_{t}=\displaystyle\left\{u\in
H^{1}(\R^{N}),\,\,\,\,
\|u\|^{2^{*}}_{L^{2^{*}}(\R^{N}_{+})}=t,\,\,\,\,\|u\|^{2^{*}}_{L^{2^{*}}(\R^{N}_{-})}=1-t\right\},$$
$$B_{t}=\displaystyle\left\{u\in H^{1}(\R^{N}_{+}),\,\,\,\, \|u\|_{L^{2^{*}}}^{2^{*}}(\R^{N}_{+})=t\right\}$$ and $$C_{t}=\displaystyle\left\{u\in H^{1}(\R^{N}_{-}),\,\,\,\, \|u\|_{L^{2^{*}}}^{2^{*}}(\R^{N}_{-})=1-t\right\}.$$
We have
\begin{equation}
A_{t}\subset B_{t} \quad\textrm{and}\quad A_{t}\subset C_{t}.
\label{Avril2}
\end{equation}
We rewrite (\ref{Avril1}) as
\begin{eqnarray*}
S_{\alpha_1,\,\alpha_2}\geq \alpha_1\inf_{A_{t}} \int_{\R^{N}_{+}}|\nabla u|^{2}dx
+\alpha_2\inf_{A_{t}}\int_{R^{N}_{-}}|\nabla u|^{2}dx.
\end{eqnarray*}
Using (\ref{Avril2}) , we see that
\begin{eqnarray}
S_{\alpha_1,\,\alpha_2}\geq \alpha_1\inf_{A_{t}} \int_{\R^{N}_{+}}|\nabla u|^{2}dx
+\alpha_2\inf_{A_{t}}\int_{R^{N}_{-}}|\nabla u|^{2}dx\geq \alpha_1\inf_{B_{t}}\int_{\R^{N}_{+}}|\nabla u|^{2}dx
+\alpha_2\inf_{C_{t}}\int_{R^{N}_{-}}|\nabla u|^{2}dx.
\label{Avril3}
\end{eqnarray}
Or, looking at (\ref{eq2}), direct computations give that
\begin{eqnarray*}
\inf_{B_{t}}\int_{\R^{N}_{+}}|\nabla u|^{2}dx=\frac{t^{\frac{2}{2^{*}}}}{2^{\frac{2}{N}}}\,S
\end{eqnarray*}
and
\begin{eqnarray*}
\inf_{C_{t}}\int_{\R^{N}_{-}}|\nabla u|^{2}dx=\frac{(1-t)^{\frac{2}{2^{*}}}}{2^{\frac{2}{N}}}\,S.
\end{eqnarray*}
Then, (\ref{Avril3}) becomes
\begin{eqnarray*}
S_{\alpha_1,\,\alpha_2}&\geq& \alpha_1\,\, \frac{t^{\frac{2}{2^{*}}}}{2^{\frac{2}{N}}}\,S+\alpha_2 \,\, \frac{(1-t)^{\frac{2}{2^{*}}}}{2^{\frac{2}{N}}}S\\[\medskipamount]
&\geq& \frac{1}{2^{\frac{2}{N}}}\max_{t\in[0,\,1]}\displaystyle\left[\alpha_1\, t^{\frac{2}{2^{*}}}+\alpha_2\, (1-t)^{\frac{2}{2^{*}}}\right]S\\[\medskipamount]
&=&\left(\frac{\alpha_1^{\frac{N}{2}}+\alpha_2^{\frac{N}{2}}}{2}
\right)^{\frac{2}{N}}S
\end{eqnarray*}
which gives (\ref{eq3}).\\
On the other hand, we claim that
\begin{equation}
S_{\alpha_1,\,\alpha_2}\leq
\left(\frac{\alpha_1^{\frac{N}{2}}+\alpha_2^{\frac{N}{2}}}{2}
\right)^{\frac{2}{N}}S. \label{eq3'}
\end{equation}
In order to prove the previous claim, for every $x\in\R^{N}$ we denote $x=(x',x_{N})$ where $x'\in \R^{N-1}$.\\
Let $\{u_{j}^{+}\}$ be a minimizing sequence of $S^{+}$. We define the sequence
 $$u_{j}^{-}(x',\,x_{N})=u_{j}^{+}(x',\,-x_{N})\quad \textrm{for all $x\in \R^{N-1}\times ]-\infty,\,0]$ and for all $j\in \N$.}$$
 Easily we see that $\{u_{j}^{-}\}$ is a minimizing sequence of $S^{-}$.\\
There exists
$t_{0}=\displaystyle\frac{(\frac{\alpha_1}{\alpha_2})^{\frac{N}{2}}}{1+(\frac{\alpha_1}{\alpha_2})^{\frac{N}{2}}}$
such that
\begin{equation*}
\begin{array}{lll}
\displaystyle
\displaystyle\sup_{t\in[0,\,1]}\frac{(\alpha_1\,t^{\frac{2}{2^{*}}}+\alpha_2\,(1-t)^{\frac{2}{2^{*}}})\,
S}{2^{\frac{2}{N}}}
=\displaystyle\frac{(\alpha_1\,t_{0}^{\frac{2}{2^{*}}}+\alpha_2\,(1-t_{0})^{\frac{2}{2^{*}}})\,
S}{2^{\frac{2}{N}}}
=\left(\frac{\alpha_1^{\frac{N}{2}}+\alpha_2^{\frac{N}{2}}}{2}
\right)^{\frac{2}{N}}S.
\end{array}
\end{equation*}
 We define the following functions :
 \begin{eqnarray*}
 v_{j}^{+}(x',\,x_{N})=t_{0}\,u_{j}^{+}(x',\,t_{0}^{\theta}\,x_{N})\,\,&\textrm{for all $x\in \R^{N-1}\times ]0,\,+\infty]$}\\[\medskipamount]
  v_{j}^{-}(x',\,x_{N})=t_{0}\,u_{j}^{-}(x',\,(1-t_{0})^{\theta}\,x_{N})\,\,&\textrm {for all $x\in \R^{N-1}\times ]-\infty,\,0]$},
\end{eqnarray*}
where $\theta=\displaystyle\frac{2(2^{*}-1)}{2^{*}}$.
Now, consider
\begin{equation}
w_{j}(x',\,x_{N})=\left\{\begin{array}{lll}
v_{j}^{+}(x',\,x_{N})\,\,&\textrm{for all $x'\in\R^{N-1}$ and for all $x_{N}\geq 0$}\\[\medskipamount]
v_{j}^{-}(x',\,x_{N})\,\,&\textrm{for all $x'\in\R^{N-1}$ and for all $x_{N}\leq 0$}
\end{array}
\right.
\end{equation}
An easy computation yields that, for large $j$, $w_{j}$ is a testing function for $S_{\alpha_1,\,\alpha_2}$ defined by (\ref{eq4}).\\
Therefore
\begin{eqnarray*}
S_{\alpha_1,\,\alpha_2}\leq \alpha_1 \int_{\R^{N}_{+}}|\nabla v_{j}^{+}|^{2}dx+\alpha_2 \int_{\R^{N}_{-}}|\nabla v_{j}^{-}|^{2}dx.
 \end{eqnarray*}
 Using the definitions of $v_{j}^{+}$ and $v_{j}^{-}$, we obtain
 \begin{equation*}
 S_{\alpha_1,\,\alpha_2}\leq \alpha_1\,\, \frac{t_{0}^{\frac{2}{2^{*}}}}{2^{\frac{2}{N}}}\,S+\alpha_2 \,\, \frac{(1-t_{0})^{\frac{2}{2^{*}}}}{2^{\frac{2}{N}}}S+o(1).
 \end{equation*}
 Then, using the definition of $t_{0}$ and letting $j\rightarrow +\infty$, we obtain
 \begin{eqnarray*}
S_{\alpha_1,\,\alpha_2}\leq\left(\frac{\alpha_1^{\frac{N}{2}}+\alpha_2^{\frac{N}{2}}}{2}
\right)^{\frac{2}{N}}S,
\end{eqnarray*}
 which gives (\ref{eq3'}).\\
 Finally, (\ref{eq3}) and (\ref{eq3'}) give the conclusion of Theorem~{\ref{th1}}.
\end{proof1}

The proof of Theorem~{\ref{th2}} follows from the following two Lemmas.
\begin{lemma}
Following the hypothesis of Theorem~{\ref{th2}}, we have
if $\alpha_{1} S<S(p)<S_{\alpha_1,\,\alpha_2}$ then the infimum in (\ref{eq5}) is achieved.
\label{lm1}
\end{lemma}
\begin{proof}~\\
We follow the arguments of Brezis-Nirenberg (\cite{BN}, proof of Lemma 2.1) {and we rely on some idea of Demyanov-Nazarov (\cite{DN}, proof of Proposition 1.1)}.
Let $\{u_{j}\}\subset H_{0}^{1}(\Omega)$ be a minimizing sequence
for (\ref{eq5}) that is,
\begin{equation}
\displaystyle\int_{\Omega}p(x)|\nabla u_{j}|^{2}dx=S(p)+o(1),
\label{eqj1}
\end{equation}~
\begin{equation}
\|u_{j}\|_{L^{2^{*}}}=1, \label{eqj2}
\end{equation}
and
\begin{equation}
\beta(u_{j})\in \Gamma.
\label{april1}
\end{equation}
Easily we see that $\{u_{j}\}$ is bounded in
$H_{0}^{1}(\Omega)$, we may extract a subsequence still denoted by
$u_{j}$, such that
\begin{eqnarray*} u_{j}\rightharpoonup u &\textrm{weakly in
$H_{0}^{1}(\Omega)$},\\[\medskipamount]
u_{j}\rightarrow u &\textrm{strongly in $L^{2}(\Omega)$},\\[\medskipamount]
u_{j}\rightarrow u&\textrm{a.e. on $\Omega$},
\end{eqnarray*}
with
$\|u\|_{L^{2^{*}}}\leq {1}$. Set $v_{j}=u_{j}-u$, so that
\begin{eqnarray*}
v_{j}\rightharpoonup {0}&\textrm{weakly in
$H_{0}^{1}(\Omega)$}\\[\medskipamount]
v_{j}\rightarrow 0 &\textrm{strongly in $L^{2}(\Omega)$},\\[\medskipamount]
v_{j}\rightarrow{0}&\textrm{a.e. on $\Omega$}.
\end{eqnarray*}
Using (\ref{eqj1}) we write
\begin{equation} \int_{\Omega}p(x)|\nabla u(x)|^{2}dx+\int_{\Omega}p(x)|\nabla
v_{j}(x)|^{2}dx=S(p)+o(1),
\label{eqj3} \end{equation}
 since $v_{j}\rightharpoonup {0}$ weakly in
$H_{0}^{1}(\Omega)$.\\
On the other hand, it follows from a result of {Brezis-Lieb
(\cite{BL}, relation (1))} that
$$
\|u+v_{j}\|_{L^{2^{*}}}^{2^{*}}=\|u\|_{L^{2^{*}}}^{2^{*}}+\|v_{j}\|_{L^{2^{*}}}^{2^{*}}+o(1),
$$
(which holds since $v_{j}$ is bounded in $L^{2^{*}}$ and
$v_{j}\rightarrow 0$ a.e.). Thus, by (\ref{eqj2}), we have
\begin{equation}
1=\|u\|_{L^{2^{*}}}^{2^{*}}+\|v_{j}\|_{L^{2^{*}}}^{2^{*}}+o(1)
\label{eqj'4}
\end{equation}
and therefore
\begin{equation}
1\leq\|u\|_{L^{2^{*}}}^{2}+\|v_{j}\|_{L^{2^{*}}}^{2}+o(1).
\label{eqj4}
\end{equation}
Using the definition of $S_{\alpha_1,\,\alpha_2}$, extending $v_{j}$ by $0$ in $\R^{N}$ (still denoted by $v_{j}$) we obtain
\begin{eqnarray}
\|v_{j}\|_{L^{2^{*}}}^{2}&\leq& \displaystyle
\frac{1}{S_{\alpha_1,\,\alpha_2}}\left[\alpha_1
\int_{\R_{+,\,x_{0}}^{N}}|\nabla
v_{j}(x)|^{2}dx+\alpha_2\int_{\R_{-,\,x_{0}}^{N}}|\nabla v_{j}(x)|^{2}dx \right]\nonumber\\[\medskipamount]
&\leq& \displaystyle \frac{1}{S_{\alpha_1,\,\alpha_2}}\left[\alpha_1
\int_{\Omega\cap\R_{+,\,x_{0}}^{N}}|\nabla
v_{j}(x)|^{2}dx+\alpha_2\int_{\Omega\cap\R_{-,\,x_{0}}^{N}}|\nabla v_{j}(x)|^{2}dx \right]\nonumber\\[\medskipamount]
&\leq& \displaystyle \frac{1}{S_{\alpha_1,\,\alpha_2}}\left[\alpha_1
\int_{\Omega}|\nabla
v_{j}(x)|^{2}dx+\alpha_2\int_{\Omega}|\nabla v_{j}(x)|^{2}dx \right]\nonumber\\[\medskipamount]
\|v_{j}\|_{L^{2^{*}}}^{2}&\leq& \displaystyle \frac{1}{S_{\alpha_1,\,\alpha_2}}\int_{\Omega}p(x)|\nabla v_{j}(x)|^{2}dx.
\label{eqj5}
\end{eqnarray}
where
$\R_{+,\,x_{0}}^{N}=\displaystyle\{x=(x',x_{N})\in\R^{N},\,\,\,\setminus\,\,x'\in
\R^{N-1},\,\, x_{N}>x_{0N}\}$ and
$\displaystyle\R_{-,\,x_{0}}^{N}=\displaystyle\{x=(x',x_{N})\in\R^{N},\,\,\,\setminus\,\,x'\in
\R^{N-1},\,\, x_{N}<x_{0N}\}$ with $x_{ON}$ is such that $\displaystyle x_{0}=\displaystyle(x',x_{0N})$.\\
We claim that $u\not\equiv 0$.\\
Indeed,  suppose that $u\equiv 0$. From (\ref{eqj3}) we obtain
\begin{equation*}
\displaystyle \int_{\Omega}p(x)|\nabla v_{j}|^{2}dx=S(p)+o(1),
\end{equation*}
then
\begin{equation*}
\displaystyle \lim_{j\rightarrow +\infty}\int_{\Omega}p(x)|\nabla v_{j}|^{2}dx=S(p).
\end{equation*}
From (\ref{eqj'4}) we see that
\begin{equation*}
\displaystyle \lim_{j\rightarrow +\infty}\|v_{j}\|_{L^{2^{*}}}=1.
\end{equation*}
Or (\ref{eqj5}) gives that
\begin{equation*}
\displaystyle\|v_{j}\|_{L^{2^{*}}}^{2} S_{\alpha_1,\,\alpha_2}\leq \int_{\Omega}p(x)|\nabla v_{j}|^{2}dx.
\end{equation*}
Passing to limit in the previous inequality we obtain $S_{\alpha_1,\,\alpha_2}\leq S(p)$. This contradicts the hypothesis $S(p)<S_{\alpha_1,\,\alpha_2}$. Consequently $u\not\equiv 0$.\\
Now, we deduce from (\ref{eqj4}) and (\ref{eqj5}) that \begin{equation}
S(p)\leq
\displaystyle S(p)\|u\|_{L^{2^{*}}}^{2}+\frac{S(p)}{S_{\alpha_1,\,\alpha_2}}\int_{\Omega}p(x)|\nabla v_{j}(x)|^{2}dx+o(1). \label{eqj6} \end{equation}
Combining (\ref{eqj3}) and (\ref{eqj6}) we obtain
\begin{equation*}
\begin{array}{ll}
\displaystyle\int_{\Omega}p(x)|\nabla u(x)|^{2}+\int_{\Omega}p(x)|\nabla v_{j}(x)|^{2}dx\leq
S(p)\|u\|_{L^{2^{*}}}^{2}+\displaystyle\frac{S(p)}{S_{\alpha_1,\,\alpha_2}}\int_{\Omega}p(x)|\nabla v_{j}(x)|^{2}dx+o(1).
\end{array}
\end{equation*}
Thus
\begin{equation*}
\begin{array}{ll}
\displaystyle\int_{\Omega}p(x)|\nabla
u(x)|^{2}dx\leq
S(p)\|u\|_{L^{2^{*}}}^{2}+\displaystyle\left[\frac{S(p)}{S_{\alpha_1,\,\alpha_2}}-1\right]\int_{\Omega}p(x)|\nabla v_{j}(x)|^{2}dx+o(1).
\end{array}
\end{equation*}
Since $S(p)<S_{\alpha_1,\,\alpha_2}$, we deduce
\begin{equation}
\displaystyle\int_{\Omega}p(x)|\nabla u(x)|^{2}dx\leq S(p)\|u\|_{L^{2^{*}}}^{2},
\label{eqj7} \end{equation}
Therefore
\begin{equation*}
\displaystyle\int_{\Omega}p(x)|\nabla u(x)|^{2}dx= S(p)\|u\|_{L^{2^{*}}}^{2}.
\end{equation*}
It follows that $u_{j}\rightarrow u$ strongly in $L^{2^{*}}(\Omega)$ and $\beta(u)\in \Gamma$. This means that $u$ is a minimum of $S(p)$.
\end{proof}
\begin{lemma}
Assume that there exists $x_{0}$ in the interior of $\Gamma$ such that the (g.c.) holds. Then
$$S(p)<S_{\alpha_1,\,\alpha_2} .$$
\label{lm2}
\end{lemma}
\begin{proof}
Let $\{\lambda_{i}(x_{0})\}_{1\leq i\leq N-1}$, denote the principal curvatures and $H(x_{0})=\displaystyle\frac{1}{N-1}\sum_{i=1}^{N-1}\lambda_{i}(x_{0})$ the mean curvature at $x_{0}$ with respect to the unit normal.\\
For the simplicity, we suppose that $x_{0}=0$. Therefore we note $\{\lambda_{i}\}_{1\leq i\leq N-1}$ the principal curvatures at $0$ and $\displaystyle H(0)=\frac{1}{N-1}\sum_{i=1}^{N-1}\lambda_{i}$. Let $R>0$,
 such that
\begin{equation*}
B(R)\cap \Omega_{1}=\left\{(x',\,x_{N})\in B(R);\,\,x_{N}> \rho(x')\right\}
\label{eq6}
\end{equation*}
\begin{equation*}
B(R)\cap \Omega_{2}=\left\{(x',\,x_{N})\in B(R);\,\,x_{N}< \rho(x')\right\}
\label{eq7}
\end{equation*}
\begin{equation*}
B(R)\cap \Gamma=\left\{(x',\,x_{N})\in B(R);\,\,x_{N}=\rho(x')\right\}
\label{eq8}
\end{equation*}
where $x'=(x_{1},x_{2},...,x_{N-1})$ and $\rho(x')$ is defined by
\begin{equation*}
\rho(x')=\Sigma_{i=1}^{N-1}\lambda_{i}x_{i}^{2}+O(|x'|^{3}).
\end{equation*}\\
We note that the condition (g.c.) implies that $\rho(x')\geq 0$.\\
Let us define, for $\varepsilon>0$ and for $t\in ]0,\,1[$ the function
\begin{equation*}
u_{0,\varepsilon, t}(x)=\left\{\begin{array}{lll}
\frac{\varphi(x)}{(\varepsilon +
{\bf{|x'|}}^{2}+t^{-\frac{N-2}{2}}{x_{N}}^{2})^\frac{{\bf{N}}-2}{2}} &\textrm{If $x_{N}> {0}$}\\[\bigskipamount]
\frac{\varphi(x)}{(\varepsilon +
{\bf{|x'|}}^{2}+(1-t)^{-\frac{N-2}{2}}{x_{N}}^{2})^\frac{{\bf{N}}-2}{2}} &\textrm{If $x_{N}< {0}$}
\end{array}
\right.
\end{equation*}
where $\varphi$ is a radial $C^{\infty}$-function such that
\begin{equation*}
\varphi(x)=\left\{\begin{array}{lll}
1 &\textrm{if $|x|\leq \frac{R}{4}$}\\[\medskipamount]
0 &\textrm{if $|x|\geq \frac{R}{2}$}.
\end{array}
\right.
\end{equation*}
There exists
$t_{0}=\displaystyle\frac{(\frac{\alpha_1}{\alpha_2})^{\frac{N}{2}}}{1+(\frac{\alpha_1}{\alpha_2})^{\frac{N}{2}}}$
such that
\begin{equation*}
\begin{array}{lll}
\displaystyle
\displaystyle\sup_{t\in[0,\,1]}\frac{(\alpha_1\,t^{\frac{2}{2^{*}}}+\alpha_2\,(1-t)^{\frac{2}{2^{*}}})\,
S}{2^{\frac{2}{N}}}
=\displaystyle\frac{(\alpha_1\,t_{0}^{\frac{2}{2^{*}}}+\alpha_2\,(1-t_{0})^{\frac{2}{2^{*}}})\,
S}{2^{\frac{2}{N}}}
=\left(\frac{\alpha_1^{\frac{N}{2}}+\alpha_2^{\frac{N}{2}}}{2}
\right)^{\frac{2}{N}}S.
\end{array}
\end{equation*}
We note $u_{0,\varepsilon}(x)=u_{0,\varepsilon, t_{0}}(x)$. Set, for $i\in\,\{1,\,2\}$,
\begin{equation*}
Q_{i}(u)=\frac{\displaystyle\int_{\Omega_{i}}\alpha_i|\nabla u|^{2}dx}{\displaystyle\left(\int_{\Omega}|u|^{2^{*}}dx\right)^{\frac{2}{2^{*}}}}
\end{equation*}
and
\begin{equation*}
Q(u)= Q_{1}(u)+ Q_{2}(u).
\end{equation*}
In order to obtain the result of Lemma~{\ref{lm2}}, we use $u_{0,\varepsilon}$ as a test function for $S(p)$.\\
From (\cite{AM}, page 13), direct computation gives
\begin{equation}
Q_{1}(u_{0,\,\varepsilon})=\left\{\begin{array}{lll} \frac{\alpha_1 t_{0}^{\frac{2}{2^{*}}}\, S}{2^{\frac{2}{N}}}+ \,\alpha_1 \,S H(0)\, A(N)\,\varepsilon^{\frac{1}{2}}|\ln(\varepsilon)| +O(\varepsilon^{\frac{1}{2}})\quad &\textrm{if $N=3$}\\[\medskipamount]
\frac{\alpha_1 t_{0}^{\frac{2}{2^{*}}}\, S}{2^{\frac{2}{N}}}+ \,\alpha_1\,S H(0)\,
A(N)\,\varepsilon^{\frac{1}{2}} + O(\varepsilon
|\ln(\varepsilon)|)\quad &\textrm{if $N\geq 4$}
\end{array}
\right.
\label{eq9}
\end{equation}
and
\begin{equation}
Q_{2}(u_{0,\,\varepsilon})=\left\{\begin{array}{lll} \frac{\alpha_2(1-t_{0})^{\frac{2}{2^{*}}}\, S}{2^{\frac{2}{N}}}-\,  \alpha_1\,S H(0)\, A(N)\,\varepsilon^{\frac{1}{2}}|\ln(\varepsilon)| +O(\varepsilon^{\frac{1}{2}})\quad &\textrm{if $N=3$}\\[\medskipamount]
\frac{\,\alpha_2(1-t_{0})^{\frac{2}{2^{*}}}\, S}{2^{\frac{2}{N}}}-\,\alpha_2\,S H(0)\,
A(N)\,\varepsilon^{\frac{1}{2}} +O(\varepsilon
|\ln(\varepsilon)|)\quad &\textrm{if $N\geq 4$}
\end{array}
\right.
\label{eq10}
\end{equation}
where $A(N)$ is a positive constant.\\
Combining (\ref{eq9}) and (\ref{eq10}) we see that,
\begin{equation*}
Q(u_{0,\,\varepsilon})=\left\{\begin{array}{lll} \frac{(\alpha_1\,t_{0}^{\frac{2}{2^{*}}}+\alpha_2\,(1-t_{0})^{\frac{2}{2^{*}}})\, S}{2^{\frac{2}{N}}}-\,(\alpha_2-\alpha_1)H(0)\,S\, A(N)\,\varepsilon^{\frac{1}{2}}|\ln(\varepsilon)|+O(\varepsilon^{\frac{1}{2}}) &\textrm{if $N=3$}\nonumber\\[\medskipamount]
\frac{(\alpha_1\,t_{0}^{\frac{2}{2^{*}}}+\alpha_2\,(1-t_{0})^{\frac{2}{2^{*}}})\,
S}{2^{\frac{2}{N}}}-\,(\alpha_2-\alpha_1)H(0)\,S\,
A(N)\,\varepsilon^{\frac{1}{2}}+O(\varepsilon |\ln(\varepsilon)|)
&\textrm{if $N\geq 4$}.
\end{array}
\right.
\end{equation*}
Therefore, using the definition of $t_{0}$, we obtain
\begin{equation*}
Q(u_{0,\,\varepsilon})\leq\left\{\begin{array}{lll} \left(\frac{\alpha_1^{\frac{N}{2}}+\alpha_2^{\frac{N}{2}}}{2} \right)^{\frac{2}{N}}S-\,(\alpha_2-\alpha_1)H(0)\,S\, A(N)\,\varepsilon^{\frac{1}{2}}|\ln(\varepsilon)|+O(\varepsilon^{\frac{1}{2}}) &\textrm{if $N=3$}\nonumber\\[\bigskipamount]\vspace{2mm}
\left(\frac{\alpha_1^{\frac{N}{2}}+\alpha_2^{\frac{N}{2}}}{2}
\right)^{\frac{2}{N}}S-\,(\alpha_2-\alpha_1)H(0)\,S\,
A(N)\,\varepsilon^{\frac{1}{2}}+O(\varepsilon |\ln(\varepsilon)|)
&\textrm{if $N\geq 4$}.
\end{array}
\right.
\label{eq11}
\end{equation*}
Finally, since $\alpha_1 <\alpha_2\,$ then we obtain the desired result.
\end{proof}
\begin{remark}(see \cite{AM}):
By looking at the previous proof, it follows that we can
relax the condition (g.c.) by allowing some of the $\lambda_{i}$'s
to be negative with mean curvature positive.
\end{remark}

\end{document}